\documentclass{amsart}
\usepackage{amsmath}
\usepackage{amssymb}
\usepackage{latexsym}
\usepackage{verbatim}
\usepackage{graphics}
\usepackage{multirow}% http://ctan.org/pkg/multirow
\usepackage{amsfonts}
\usepackage{graphicx}
\usepackage{epstopdf}

\usepackage{etoolbox} %i added this to use \bbordermatrix in theorem 8
\let\bbordermatrix\bordermatrix % i added this too
\patchcmd{\bbordermatrix}{8.75}{4.75}{}{}% i added this too
\patchcmd{\bbordermatrix}{\left(}{\left[}{}{}% i added this too
\patchcmd{\bbordermatrix}{\right)}{\right]}{}{}% i added this too

\usepackage{blkarray}
\usepackage{multirow}% http://ctan.org/pkg/multirow

\theoremstyle{plain}
\newtheorem{thm}{Theorem}

\newtheorem{lem}{Lemma}
\newtheorem{cor}{Corollary}

\def\R{\mathbb{R}}
\def\Z{\mathbb{Z}}
\def\e{\varepsilon}

\def\x{\times}
\def\f{\phi}
\def\a{\alpha}
\def\b{\beta}
\def\g{\gamma}
\def\o{\omega}
\def\ord{{\rm ord}}

\theoremstyle{remark}
\newtheorem*{defn}{Definition}
\newtheorem{question}{Question}
\newtheorem{example}{Example}
\newtheorem*{rem}{Remark}

\begin{document}

\title{Colorings, determinants and Alexander polynomials for spatial graphs}
\date{\today}

\author{Terry Kong}
\address{Stanford University}
\email{tckong@stanford.edu}
	
\author{Alec Lewald}
\address{Cal Poly Pomona}
\email{aleclewald88@yahoo.com}

\author{Blake Mellor}
\address{Mathematics Department\\
   	Loyola Marymount University\\
   	Los Angeles, CA  90045-2659}
\email{blake.mellor@lmu.edu}

\author{Vadim Pigrish}
\address{University of Pennsylvania}
\email{vpigrish@sas.upenn.edu}

\keywords{spatial graphs, p-colorings, Alexander polynomial}
\subjclass[2000]{05C10; 57M25}

\thanks{This research was supported in part by NSF grant DMS-0905687, and by the Summer Undergraduate Research Program at LMU}

\begin{abstract}
A {\em balanced} spatial graph has an integer weight on each edge, so that the directed sum of the weights at each vertex is zero.  We describe the Alexander module and polynomial for balanced spatial graphs (originally due to Kinoshita \cite{ki}), and examine their behavior under some common operations on the graph.  We use the Alexander module to define the determinant and $p$-colorings of a balanced spatial graph, and provide examples. We show that the determinant of a spatial graph determines for which $p$ the graph is $p$-colorable, and that a $p$-coloring of a graph corresponds to a representation of the fundamental group of its complement into a metacyclic group $\Gamma(p,m,k)$.  We finish by proving some properties of the Alexander polynomial.
\end{abstract}

\maketitle

%\tableofcontents

%\setlength{\baselineskip}{20pt}

\section{Introduction} \label{S:intro}

In the study of knots, there are close connections among the Alexander polynomial \cite{al} of a knot, the determinant of a knot (see \cite{li}, for example) and the $p$-colorings of a knot \cite{cf}.  The main purpose of this paper is to establish analogous connections among the corresponding invariants of a spatial graph.  Our spatial graphs will come equipped with both an orientation and an integer weight on each edge.

The Alexander module and Alexander polynomial for a spatial graph were first defined in the 1950's by Kinoshita \cite{ki}.  Ishii and Yasuhara \cite{iy} introduced $p$-colorings of (Eulerian) spatial graphs.  McAtee, Silver and Williams \cite{msw} extended Ishii and Yasuhara's work to define a more general coloring theory, along with a topological interpretation in terms of the homology of cyclic branched covers, and related it to Litherland's Alexander polynomial for $\theta$-graphs \cite{lit}.

We will use Kinoshita's version of the Alexander module to define a theory of $p$-colorings of a spatial graph.  Along the way, we will define a sequence of {\em determinants} of a spatial graph, derived from the Alexander module.  For knots these determinants are the same as the sequence of Alexander polynomials (the classical knot determinant is simply the absolute value of the first Alexander polynomial at $-1$).  For spatial graphs, however, the determinants are stronger invariants than the Alexander polynomial.  We will show that these determinants provide the obstructions to finding $p$-colorings of spatial graphs, as the Alexander polynomials do for knots.  We will also show that $p$-colorings can be interpreted in terms of metacyclic representations of the fundamental group of the knot complement.

In section \ref{S:alexander}, we describe the Alexander module and the Alexander polynomials of a spatial graph (more precisely, a {\em balanced} spatial graph, where the directed sum of the edge weightings at each vertex is zero).  In section \ref{S:determinant} we introduce the determinants of a spatial graph, and define what it means for a spatial graph to be $p$-colorable.  We prove that the determinants provide obstructions to coloring the graph.  We also show that $p$-colorings can be interpreted as representations of the fundamental group of the graph complement into certain metacyclic groups, and count these representations.  In Section \ref{S:properties} we explore how the Alexander polynomial is affected by common operations such as reversing orientation, contracting an edge of the graph, and joining two graphs at a vertex.  Finally, we pose some questions for further research.

\section{Alexander polynomial for spatial graphs} \label{S:alexander}

In this section we will study the {\em Alexander module} and {\em Alexander polynomials} of a spatial graph (more specifically, a {\em balanced} spatial graph, as defined below).  These are essentially the same as the Alexander module and polynomials defined by Kinoshita \cite{ki}, though we will present them slightly differently.  

\subsection{Alexander module} \label{SS:module}

An {\em oriented spatial graph} $G$ is the image of an embedding of an abstract directed graph $\Gamma$ in $\R^3$.  A {\em diagram} for $G$ is a regular projection of $G$ to a plane in $\R^3$. We will consider oriented spatial graphs with a {\it weighting} $\omega$ which assigns to each edge of $G$ an integer weight. We say that a weighting $\omega$ is {\it balanced} if at each vertex, the sum of the weights of the edges directed into the vertex is equal to the sum of the weights of the edges directed out from the vertex.  A {\em balanced} spatial graph is a pair $(G, \omega)$ of an oriented spatial graph $G$ and a balanced weighting $\omega$.  A balanced weighting $\omega$ is {\em reduced} if the greatest common divisor of all the weights is 1; we will see that we can usually restrict our attention to reduced weightings.

Balanced weightings can be interpreted topologically.  Kinoshita \cite{ki} viewed a balanced spatial graph as a cycle in the module $C_1$ of 1-chains in $S^3$ over $\Z$.  Alternatively, we can view a weighting $\omega$ as a function from the meridian of each edge of the graph to $\Z$.  The meridians generate $H_1(S^3-G; \Z)$, the first homology of the graph exterior.  The relations in the first homology group are exactly the oriented sums of the generators for the edges incident to each vertex of $G$.  Since $\omega$ sends these sums to 0, it is an element of Hom$(H_1(S^3-G; \Z), \Z) \cong H^1(S^3-G; \Z)$, and hence can be interpreted as a cohomology class.

Every oriented graph has a trivial balanced weighting where every edge has weight 0.  McAtee, Silver and Williams \cite{msw} showed that if a graph has a specified spanning tree and any orientation on the edges outside the spanning tree, then the spanning tree can be oriented to allow a non-trivial balanced weighting where all weights are non-negative (i.e. all edges not in the spanning tree receive weight 1, and edges in the spanning tree are oriented and weighted so that the final result is balanced).  Given a balanced weighting of an oriented graph $G$, changing the orientation of an edge and the sign of its weight gives another balanced weighting of an oriented graph $G'$ (where $G'$ differs from $G$ in the orientation of one edge).  So McAtee, Silver and Williams' construction shows that every oriented graph has a nontrivial balanced weighting (though we may not be able to ensure that all the weights are non-negative, such as with a $\theta$-graph with all edges directed into the same vertex).

\begin{figure}[htbp]
\begin{center}
\scalebox{.8}{\includegraphics{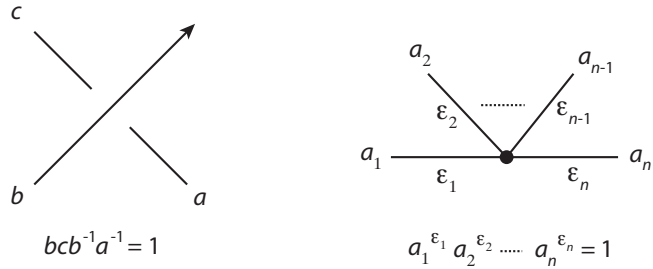}}
\end{center}
\caption{Wirtinger relations for $\pi_1(S^3 - G)$.}
\label{F:wirtinger}
\end{figure}

The fundamental group $\pi_1(S^3-G)$ has a Wirtinger presentation constructed from a diagram $D$, where the generators correspond to the arcs in the diagram.  The presentation has a relation at each crossing and vertex in the diagram, as shown in Figure \ref{F:wirtinger}.  At a vertex, the {\it local sign} $\e_i$ of arc $a_i$ is 1 if the arc is directed into the vertex, and -1 if the arc is directed out from the vertex.

Given a balanced spatial graph $(G, \omega)$, there is a well-defined {\em augmentation homomorphism} $\chi: \pi_1(S^3 - G) \rightarrow \Z = \langle t \vert \ \rangle$ that sends each Wirtinger generator to $t^w$, where $w$ is the weight of the edge containing that arc (the weighting must be balanced for the image of the relation at each vertex to be the identity).  We also let $\chi$ denote the linear extension of this map from the group ring $\Z[\pi_1(S^3 - G)]$ to $\Z[t^{\pm 1}]$.  If $\omega$ is a {\em reduced} weighting, then $\chi$ is an epimorphism.  We let $E = S^3 - G$; there exists a space $E_\chi$ and a covering map $p: E_\chi \rightarrow E$ with $p_*(\pi_1(E_\chi)) = {\rm Ker}(\chi)$ (i.e. $E_\chi$ is the  cyclic cover of $E$ corresponding to $\chi$).  If $G$ is a connected graph with vertex set $V$, the {\it Alexander module} over $\Z[t^{\pm 1}]$ for the pair $(G, \omega)$ is the relative homology module $A(G, \omega) = H_1(E_\chi, p^{-1}(V); \Z)$.  Then there is a $\Z[t^{\pm 1}]$-exact sequence (where $\langle t-1 \rangle$ is the ideal of $\Z[t^{\pm 1}]$ generated by $t-1$) (see \cite{fo2} and \cite[Thm. 15.3.4]{ka}):
$$0 \rightarrow H_1(E_\chi; \Z) \rightarrow A(G, \omega) \rightarrow \langle t-1 \rangle \rightarrow 0$$

Given a diagram $D$ for $G$, a presentation matrix for $A(G, \omega)$ is $\left[ \chi \left( \frac{\partial r_i}{\partial a_j} \right) \right]$, where $r_i$ and $a_j$ are the relations and generators of the Wirtinger presentation (from diagram $D$), and the derivative is Fox's free derivative from $\Z[\pi_1(S^3 - G)]$ to itself.  To compute the entries of this matrix (which we denote by $M(D, \omega)$), we recall the following rules for computing the Fox derivative \cite{fo} (this is not a complete description; just the rules we will need):
\begin{enumerate}
	\item $\dfrac{\partial a_i}{\partial a_j} = \left\{ \begin{matrix} 1\ {\rm if}\ i = j\\ 0\ {\rm if}\ i \neq j \end{matrix} \right.$,
	\item $\dfrac{\partial}{\partial a_j} (gh) = \dfrac{\partial}{\partial a_j}(g) + g \dfrac{\partial}{\partial a_j}(h)$, for $g, h \in \pi_1(S^3 - G)$,
	\item $\dfrac{\partial}{\partial a_j} (g^{-1}) = -g^{-1} \dfrac{\partial}{\partial a_j} (g)$ for $g \in \pi_1(S^3 - G)$.
\end{enumerate}

In the crossing relation $bcb^{-1}a^{-1}$ we will assume arc $b$ has weight $w_1$, and arcs $a$ and $c$ (on the same edge) have weight $w_2$.  In the vertex relation, arc $a_i$ has weight $w_i$.  Let $m_i = \e_1w_1 + \e_2w_2 + \cdots + \e_{i-1}w_{i-1} + \min\{\e_i, 0\}w_i$.
\begin{align*}
	\chi \left(\frac{\partial}{\partial b} (bcb^{-1}a^{-1})\right) &= \chi\left(1 + bc(-b^{-1})\right) = 1 - t^{w_1 + w_2 - w_1} = 1-t^{w_2} \\
	\chi \left(\frac{\partial}{\partial c} (bcb^{-1}a^{-1})\right) &= \chi (b) = t^{w_1} \\
	\chi \left(\frac{\partial}{\partial a} (bcb^{-1}a^{-1})\right) &= \chi\left(bcb^{-1}(-a^{-1})\right) = -t^{w_1 + w_2 - w_1 - w_2} = -1\\
	\chi \left(\frac{\partial}{\partial a_i} (a_1^{\e_1}a_2^{\e_2} \cdots a_n^{\e_n})\right) &= \left\{ \begin{matrix} \chi \left(a_1^{\e_1}\cdots a_{i-1}^{\e_{i-1}}\right) & {\rm if}\ \e_i = 1 \\ \chi \left(a_1^{\e_1}\cdots a_{i-1}^{\e_{i-1}}(-a_i^{-1})\right) &{\rm if}\ \e_i = -1 \end{matrix} \right\} = \e_i t^{m_i}
\end{align*}

So our linear versions of the Wirtinger relations in Figure \ref{F:wirtinger}, which we will call Alexander relations, are: \medskip
\begin{itemize}
	\item[] Crossing relation: $(1-t^{w_2})b + t^{w_1}c - a = 0$ and 
	\item[] Vertex relation: $\sum_{i=1}^n{\e_i t^{m_i}a_i}  = 0$.
\end{itemize}

\medskip
\begin{rem}
We conclude this section with a few remarks about the Alexander module and matrix. \begin{enumerate}
	\item While $M(D, \omega)$ depends on the diagram $D$, $A(G, \omega) = H_1(E_\chi, p^{-1}(V); \Z)$ depends only on the fundamental group of $S^3-G$, and the choice of weighting $\omega$.  This means the Alexander module is an invariant of {\em flexible} vertex isotopy -- the order of the edges around a vertex is not fixed.
	\item If $\omega$ is a balanced weighting which is {\em not} reduced, and the weights have greatest common divisor $g > 1$, we can still carry through the procedure above, except now $\chi$ maps $\pi_1(S^3 - G)$ onto $\Z = \langle t^g\vert \ \rangle$.  Let $\omega'$ be the reduced weighting found by dividing each weight in $\omega$ by $g$.  Then $M(D, \omega) = M(D, \omega')\vert_{t \mapsto t^g}$.  So we lose no information by replacing $\omega$ with $\omega'$, and we can generally restrict ourselves to reduced weightings.
	\item If we reverse the orientation of an edge, we replace any Wirtinger generator $a_i$ along that edge with $a_i^{-1}$.  If we also change the sign of the weight of the edge, then the image of $a_i^{-1}$ under $\chi$ is the same as the image of $a_i$ with the original weight.  So changing the orientation of an edge and reversing the sign of the weight of an edge have the same effect on the Alexander module and matrix, and doing both together leaves the Alexander module and matrix unchanged.
\end{enumerate}
\end{rem}

\subsection{Alexander polynomials} \label{SS:polynomial}

The Alexander polynomials are the greatest common divisors of the generators of the elementary ideals of the Alexander module; like the module itself, these are isotopy invariants of the pair $(G, \omega)$.  The polynomials can be computed using minors of the Alexander matrix, but are well defined only up to multiplication by units in $\Z[t^{\pm 1}]$ (i.e. by $\pm t^r$).  If a spatial graph diagram $D$ has $e$ edges, $v$ vertices and $c$ crossings, then $M(D, \omega)$ is a $(c+v) \x (c+e)$ matrix.

\begin{defn}
Given a matrix $M$ (not necessarily square) where all entries are integers (or polynomials with integer coefficients), we define $\det(M,k)$ to be the greatest common divisor of all the $k\times k$ minors of $M$.  (If $k \leq 0$, we set $\det(M,k) = 1$.)
\end{defn}

\begin{defn}
The {\em $k$th Alexander polynomial} of balanced spatial graph $(G, \omega)$ with diagram $D$ is $\Delta_k(G,\omega)(t) = \det(M(D, \omega), c+v-k)$, modulo multiplication by a factor of $\pm t^r$.
\end{defn}

It is well-known that any one of the Wirtinger relations can be written as a product of the others; hence any one row of the Alexander matrix $M(D, \omega)$ is a linear combination of the others (over $\Z[t^{\pm 1}]$).  So $\Delta_0(G, \omega)(t) = 0$ for any pair $(G, \omega)$.  It is clear from the definition that, if $k < l$, then $\Delta_l(G, \omega) \vert \Delta_k(G, \omega)$.  So the most interesting of the Alexander polynomials is usually $\Delta_1(G, \omega)$ -- this is what we mean if we ever refer to {\em the} Alexander polynomial of a balanced graph $(G, \omega)$.

\begin{figure}[htbp]
\begin{center}
\scalebox{.5}{\includegraphics{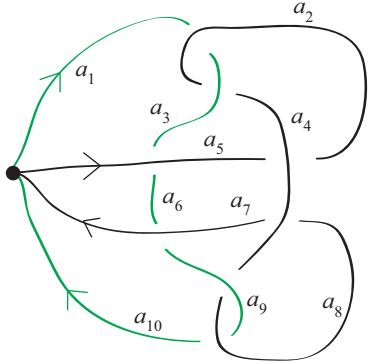}}
\end{center}
\caption{A spatial graph diagram $D$.}
\label{F:bouquet}
\end{figure}

\begin{example} \label{E:alexpoly}
Consider the spatial graph $(G, \omega)$ with diagram $D$ shown in Figure \ref{F:bouquet}, with the edges given weights $x$ and $y$ (the edge containing arc $a_1$ has weight $x$).  Then
$$M(D, \omega) = \left[ \begin{matrix} -1 & 1-t^x & t^y & 0 & 0 & 0 & 0 & 0 & 0 & 0 \\ 0 & -1 & 1-t^y & t^x & 0 & 0 & 0 & 0 & 0 & 0\\ 0 & 0 & t^y & 0 & 1-t^x & -1 & 0 & 0 & 0 & 0 \\ 0 & t^y & 0 & 1-t^y & -1 & 0 & 0 & 0 & 0 & 0\\ 0 & 0 & 0 & 0 & 0 & -1 & 1-t^x & 0 & t^y & 0 \\ 0 & 0 & 0 & 1-t^y & 0 & 0 & -1 & t^y & 0 & 0 \\ 0 & 0 & 0 & t^x & 0 & 0 & 0 & -1 & 1-t^y & 0\\ 0 & 0 & 0 & 0 & 0 & 0 & 0 & 1-t^x & t^y & -1 \\ -t^{-x} & 0 & 0 & 0 & -t^{-x-y} & 0 & t^{-x-y} & 0 & 0 & t^{-x} \end{matrix}\right]$$
If both edges are given weight 1, then $\Delta_1(G, \omega)(t) = t^2 - 2t + 2$ (normalized so that the lowest term is the constant term) and $\Delta_2(G, \omega)(t) = 1$ (so $\Delta_k(G, \omega) = 1$ for $k \geq 2$).
\end{example}

\section{Determinants and colorings of spatial graphs} \label{S:determinant}

\subsection{Determinants of spatial graphs}
The {\it determinant} of a knot is equal to the absolute value of the value of the Alexander polynomial at $-1$.  For knots, this is equivalent to replacing $t$ by $-1$ in the presentation matrix for the module and computing the determinant (after removing a row and column), since the matrix in this case is square.  However, for spatial graphs the results of these two computations can be different, since we are looking at the greatest common divisor of a set of minors.  So we will define determinants from the presentation matrix for the Alexander module, rather than the Alexander polynomial.

\begin{defn}
Let $n$ be an nonzero integer.  The {\em $k$th determinant at $n$} of a balanced spatial graph $(G, \omega)$ with diagram $D$ is $\det_k(G, \omega)(n) = \det(M(D, \omega)\vert_{t = n}, c+v-k)$, where $M(D, \omega)$ is the presentation matrix for the Alexander module, $c$ and $v$ are the numbers of crossings and vertices, and $0 \leq k \leq c+v-1$.
\end{defn}

\begin{rem} 
${}$
\begin{enumerate}
	\item When $n$ is prime (or $n = \pm 1$), the determinants at $n$ are the generators of the elementary ideals of the $\Z[1/n]$-module obtained from the Alexander module by setting $t = n$, and so do not depend on the specific diagram $D$ used.  They are well-defined up to multiplication by a power of $n$; typically, we will normalize them by factoring out the powers of $n$.
	\item For all $k$, $\det_{k+1}(G, \omega)(n) \vert \det_k(G, \omega)(n)$.
	\item $\Delta_k(G, \omega)(n) \vert \det_k(G, \omega)(n)$.  However, unlike for knots, they may not be equal (see Example \ref{E:alexdet}).
	\item We are particularly interested in the value of the determinant at $-1$, as in the knot determinant.  In this case, the absolute value of the determinant is well-defined, and the weights on the edges can be viewed as elements of $\Z_2$.
\end{enumerate}
\end{rem}

We can use the determinant to show that, as for knots, the value of the Alexander polynomial at $t = 1$ is 1.

\begin{thm} \label{T:det(1)}
Let $(G, \omega)$ be a balanced spatial graph.  Then $\det_1(G, \omega)(1) = 1$.
\end{thm}
\begin{proof}
When we set $t= 1$, the relations derived from the Wirtinger relations in Figure \ref{F:wirtinger} become:
\begin{align*}
	(1-t^{w_2})b + t^{w_1}c - a = 0 &\longrightarrow a = c \\
	\sum_{i=1}^n{\e_i t^{m_i}a_i}  = 0 &\longrightarrow \sum_{i=1}^n{\e_i a_i}  = 0
\end{align*}
These are exactly the relations for the first homology group $H_1(S^3-G; \Z)$, which is a free abelian group.  Hence the elementary divisors are trivial, so $\det_1(G, \omega)(1) = 1$.
\end{proof}

\begin{cor} \label{C:alex(1)}
For any balanced spatial graph $(G, \omega)$, $\Delta_1(G, \omega)(1) = 1$, so the sum of the coefficients is 1.
\end{cor}

%\begin{example}\label{E:determinant}
%Consider again the graph from Example \ref{E:alexpoly}, shown in Figure \ref{F:bouquet}.  In this case, the determinant at $n$ {\it does} seem to be equal to the value of the Alexander polynomial at $n$, for every choice of weighting (based on some computer experiments).  In particular, $\det_1(G, \omega)(-1) = 5$ whenever the weights on both edges are odd, and 1 otherwise.
%\end{example}

\begin{figure}[htbp]
\begin{center}
\scalebox{.5}{\includegraphics{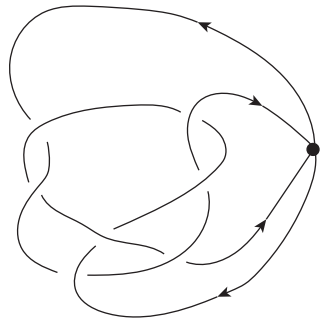}}
\end{center}
\caption{A graph $G$ for which $\Delta_1(G, \omega)(-1) \neq \det_1(G, \omega)(-1)$.}
\label{F:seven21b}
\end{figure}

\begin{example} \label{E:alexdet}
The graph $G$ in Figure \ref{F:seven21b} has $\Delta_1(G, \omega)(t) = 1$ for every choice of weighting.  However, when both edges receive weight 1, $\det_1(G, \o)(-1) = 7$, showing the graph is non-trivial.
\end{example}

\subsection{Colorings of spatial graphs} \label{SS:color}

Now we turn to colorings of spatial graphs by elements of $\Z_p$ for a prime $p$.  Given a balanced spatial graph $(G, \omega)$ with diagram $D$, the {\em coloring matrix at $n$} is
$$C(D, \omega, n) = M(D, \omega)\vert_{t = n}$$

Suppose $D$ has $c$ crossings, $e$ edges and $v$ vertices.  Given an odd prime $p$, a {\em $p$-coloring} of $D$ is an assignment of an element of $\Z_p$ to each arc (i.e. a vector in $\Z_p^{c+e}$) which is in the nullspace of $C(D, \omega, n)$ (modulo $p$).  We are interested in the number of linearly independent $p$-colorings for each $p$, so we define $N_p(G, \omega, n)$ to be the nullity of $C(D, \omega, n)$ over $\Z_p$ (this does not depend on the choice of diagram $D$).  Since $C(D, \omega, n)$ is a $(c+v) \x (c+e)$ matrix, and we know that any row is a linear combination of the others over $\Z$ (and hence over any $\Z_p$), the rank-nullity theorem tells us that $N_p(G, \omega, n) \geq e-v+1$.  We want to determine all the values of $p$ for which $N_p(G, \omega, n) > e-v+1$. We will do this using the determinants of the graph.  Our proof makes use of Lemma \ref{L:zeros}, which is simply a special case of the general fact that, in $\R^m$, an $(n+1)$-dimensional subspace and an $(m-n)$-dimensional subspace must intersect in at least a line.

\begin{lem}\label{L:zeros}
Let $S$ be a subspace of $\R^m$, where $m \geq n+1$.  $S$ contains $n+1$ linearly independent vectors $\vec v_1, \dots , \vec v_{n+1}$ if and only if for every choice of $1 \leq j_1 < j_2 <  \dots <  j_n \leq m$ there exists a non-trivial vector $\vec x$ in $S$ with $x_{j_1} = x_{j_2} = \cdots = x_{j_n} = 0$.
\end{lem}

\begin{thm} \label{T:detcolor}
A diagram for a balanced spatial graph $(G, \omega)$ will have {\em more} than $e-v+k$ linearly independent $p$-colorings at $n$ (i.e. $N_p(G, \omega, n) > e-v+k$) if and only if $p \vert \det_k(G, \omega)(n)$.
\end{thm}
\begin{proof}
Suppose that a diagram $D$ for $G$ has $e-v+k+1$ linearly independent $p$-colorings at $n$.  In other words, the nullspace for the matrix $C(D, \omega, n)$ (modulo $p$) contains $e-v+k+1$  linearly independent vectors.  By Lemma \ref{L:zeros}, for every choice of $J = \{j_1, j_2, \dots, j_{e-v+k}\}$ there is a non-trivial vector $\vec x$ with $x_{j_i} \equiv 0 \pmod{p}$ for all $1 \leq i \leq {e-v+k}$ such that $C(D, \omega, n)\vec x = \vec 0 \pmod{p}$.  This means if we remove columns $j_1, \dots, j_{e-v+k}$ from $C(D, \omega, n)$, there is still a non-trivial vector in the nullspace (modulo $p$) of the resulting $(c+v) \times (c+v-k)$ matrix $C(D, \omega, n)_J$, and hence in the nullspace of every $(c+v-k) \times (c+v-k)$ submatrix of $C(D, \omega, n)_J$.  Hence $p$ divides the determinant of each of these submatrices (since each is singular modulo $p$).  Since this is true for every choice of $J$, $p$ divides every $(c+v-k) \times (c+v-k)$ minor of $C(D, \omega, n)$, and hence divides $\det_k(G, \omega)(n)$.

Since this argument is reversible, the converse is also true.
\end{proof}

\begin{example} \label{E:coloring}
Consider again the oriented spatial graph in Figure \ref{F:bouquet}.  If both edges have weight 1, we find that $\det_1(G, \omega)(-1) = 5$ and $\det_1(G, \omega)(5) = 17$; also, $\det_2(G, \omega)(-1) = \det_2(G, \omega)(5) = 1$.  Then Theorem \ref{T:detcolor} implies that $N_5(G, \omega, -1) = 3$ while $N_p(G, \omega, -1) = 2$ for all other primes $p$.  Similarly, $N_{17}(G, \omega, 5) = 3$ and $N_p(G, \omega, 5) = 2$ for all other primes $p$.
\end{example}

%\begin{example}
%Consider again the graph in Figure \ref{F:seven21b} from Example \ref{E:alexdet}, with every edge given weight 1.  Since $\det_1(G, \omega)(-1) = 7$, we know $N_7(G, \omega, -1) > 2$.  This is the nullity over $\Z_7$ of the matrix $C(D, \omega, -1)$.  But, for any $m$, $C(D, \omega, -1 + 7m) \equiv C(D, \omega, -1) \pmod{7}$, so $N_7(G, \omega, -1 + 7m) = N_7(G, \omega, -1) > 2$.  Hence $\det_1(G, \omega)(-1+7m)$ is also a multiple of 7.  This explains the periodic behavior of $\det_1(G, \omega)$ observed in Example \ref{E:alexdet}.
%\end{example}

\begin{rem}
Ishii and Yasuhara \cite{iy} studied $p$-colorings of graphs where all vertices have even degree.  These graphs admit an Euler cycle that traverses every edge exactly once.  A choice of Euler cycle induces an orientation on the edges of the graph, such that the indegree and outdegree of each vertex are equal.  With such an orientation, we can define a balanced weighting by giving every edge weight 1.  If we then let $t = -1$, the coloring relations are exactly the relations studied by Ishii and Yasuhara.
\end{rem}

\subsection{Colorings as group representations}

It is well-known that a classical $p$-coloring of a knot (with $p$ prime) can be viewed as a representation of the knot group in the dihedral group $D_p$.  In this section we will show that a $p$-coloring at $k$ of a balanced spatial graph $(G, \omega)$ is a representation of $\pi_1(S^3-G)$ in a metacyclic group $\Gamma(p,m,k)$ (where $p$ is prime, $k$ is a positive integer, and $k^m \equiv 1 \pmod{p}$) with presentation
$$\Gamma(p,m,k) = \langle \a, \b \vert \a^p = \b^m = 1, \b\a\b^{-1} = \a^k\rangle$$
In particular, $G(p,2,-1) = D_p$.  Any element of $\Gamma(p,m,k)$ can be uniquely written as $\a^a \b^b$, with $0 \leq a \leq p-1$ and $0 \leq b \leq m-1$, so $\vert \Gamma(p,m,k)\vert = pm$.  Also, the relations in $\Gamma(p,m,k)$ imply
$$\a = \b^m \a \b^{-m} = \a^{k^m}.$$
This explains the requirement that $k^m \equiv 1 \pmod{p}$.  In other words, $m$ must be a multiple of the order of $k$ in the multiplicative group $\Z_p^*$ (denoted $\ord_p(k)$).  Also note that the final relation can be rewritten in several ways (here $1/k$ is the reciprocal of $k$ in the field $\Z_p^*$):
$$\b\a = \a^k\b \qquad \a\b^{-1} = \b^{-1}\a^k \qquad \a\b = \b\a^{1/k} \qquad \b^{-1}\a = \a^{1/k}\b^{-1}$$

Suppose $G$ is an oriented spatial graph with diagram $D$, from which we obtain a Wirtinger presentation for $\pi_1(S^3-G)$.  Any homomorphism (i.e. representation) $\phi: \pi_1(S^3 - G) \rightarrow \Gamma(p,m,k)$ is defined by $\phi(a) = \a^{\g(a)}\b^{\omega(a)}$ for each arc $a$ of $D$ (also a generator of the Wirtinger presentation), where $\g$ and $\omega$ are some maps from the generators of $\pi_1(S^3 - G)$ to elements of $\Z_p$ and $\Z_m$, respectively.

\begin{thm} \label{T:representation}
If $\phi: \pi_1(S^3 - G) \rightarrow \Gamma(p,m,k)$ is a homomorphism defined by $\phi(a) = \a^{\g(a)}\b^{\omega(a)}$ for each generator $a$ of the Wirtinger presentation of $\pi_1(S^3 - G)$, then $\omega$ is a balanced weighting (modulo $m$) of the edges of $G$, and $\g$ is a $p$-coloring of $(G, \omega)$ at $k$.  Conversely, given a balanced weighting $\omega$ of $G$ and $p$-coloring $\g$ of $(G, \omega)$ at $k$, $\phi(a) = \a^{\g(a)}\b^{\omega(a)}$ is a well-defined representation of $\pi_1(S^3-G)$ into $\Gamma(p,m,k)$.
\end{thm}
\begin{proof}
We need to look at the images of the relations in $\pi_1(S^3-G)$ under the action of $\phi$.  We first consider the crossing relation $bcb^{-1}a^{-1} = 1$ from Figure \ref{F:wirtinger}.  The image of the crossing relation under $\phi$ is:
\begin{align*}
	\phi(bcb^{-1}a^{-1}) &= \a^{\g(b)}\b^{\o(b)}\a^{\g(c)}\b^{\o(c)}\b^{-\o(b)}\a^{-\g(b)}\b^{-\o(a)}\a^{-\g(a)} \\
	%&= \a^{\g(b)}\a^{k^{\o(b)}\g(c)}\b^{\o(b)}\b^{\o(c)}\b^{-\o(b)}\a^{-\g(b)}\b^{-\o(a)}\a^{-\g(a)} \\
	%&= \a^{\g(b)+k^{\o(b)}\g(c)}\b^{\o(c)}\a^{-\g(b)}\b^{-\o(a)}\a^{-\g(a)} \\
	%&= \a^{\g(b)+k^{\o(b)}\g(c)}\a^{-k^{\o(c)}\g(b)}\b^{\o(c)}\b^{-\o(a)}\a^{-\g(a)} \\
	%&= \a^{\g(b)+k^{\o(b)}\g(c)-k^{\o(c)}\g(b)}\b^{\o(c)-\o(a)}\a^{-\g(a)} \\
	%&= \a^{\g(b)+k^{\o(b)}\g(c)-k^{\o(c)}\g(b)}\a^{-k^{\o(c)-\o(a)}\g(a)}\b^{\o(c)-\o(a)} \\
	&= \a^{\g(b)+k^{\o(b)}\g(c)-k^{\o(c)}\g(b)-k^{\o(c)-\o(a)}\g(a)}\b^{\o(c)-\o(a)} 
\end{align*}

If $\f$ is a homomorphism, the image of the relation is trivial, so the exponents of $\a$ and $\b$ are both 0.  Then $\o(a) \equiv \o(c) \pmod{m}$ (and, more generally, this is true whenever $a$ and $c$ are arcs on the same edge).  So $\o$ is a mod-$m$ weighting of the edges of $G$.  Also, $\g(b)+k^{\o(b)}\g(c)-k^{\o(c)}\g(b)-k^{\o(c)-\o(a)}\g(a) = (1-k^{\o(a)})\g(b)+k^{\o(b)}\g(c)-\g(a) = 0$, which is the coloring relation at the crossing.

Now we consider a relation $\prod_{i=1}^n{a_i^{\e_i}} = 1$ (see Figure \ref{F:wirtinger}), where $\e_i$ is the local sign of arc $a_i$ at the vertex.  Suppose that $\o(a_i) = w_i$ and $\g(a_i) = x_i$.  The image of the vertex relation under $\phi$ is 
$$\phi\left(\prod_{i=1}^n{a_i^{\e_i}} \right) = \prod_{i=1}^n{(\a^{x_i}\b^{w_i})^{\e_i}}$$
Note that 
$$(\a^{x_i}\b^{w_i})^{-1} = \b^{-w_i}\a^{-x_i} = \a^{-x_i/k^{w_i}}\b^{-w_i} = \a^{-x_i k^{-w_i}}\b^{-w_i}$$
So we can write 
$$(\a^{x_i}\b^{w_i})^{\e_i} = \a^{\e_i x_i k^{\min(\e_i, 0) w_i}}\b^{\e_i w_i}$$
Also recall that $\b^{\e}\a = \a^{k^{\e}}\b$, where $\e = \pm 1$.  Then we have (where $m_i = \e_1w_1 + \cdots + \e_{i-1}w_{i-1} + \min(\e_i,0)w_i$)
\begin{align*}
	\phi\left(\prod_{i=1}^n{a_i^{\e_i}} \right) &= \prod_{i=1}^n{\a^{\e_i x_i k^{\min(\e_i, 0) w_i}}\b^{\e_i w_i}}\\
	&= \a^{\e_1 x_1 k^{\min(\e_1, 0) w_1}}\a^{\e_2 x_2 k^{\e_1 w_1 + \min(\e_2, 0) w_2}}\a^{\e_3 x_3 k^{\e_1 w_1 + \e_2 w_2 + \min(\e_3, 0) w_3}}\cdots \a^{\e_n x_n k^{m_n}} \b^{\sum{\e_i w_i}} \\
	&= \a^{\sum{\e_i x_i k^{m_i}}} \b^{\sum{\e_i w_i}} = 1
\end{align*}
This means that $\sum{\e_i w_i} \equiv 0 \pmod{m}$, so $\o$ is a {\em balanced} weighting of $G$ modulo $m$.  Also, $\sum_{i=1}^n{\e_i k^{m_i}x_i} \equiv 0 \pmod{p}$, so the coloring relations at the vertices are satisfied.  Hence $\o$ is a balanced weighting of $G$ (modulo $m$), and $\g$ is a $p$-coloring of $(G, \o)$ at $k$.

Conversely, if we start with a balanced weighting $\o$ and a coloring $\g$, the images of the relations in $\pi_1(S^3-G)$ under $\phi$ are trivial, 
so $\phi$ is a well-defined representation into $\Gamma(p,m,k)$.
\end{proof}

In particular, $p$-colorings at $k = -1$ correspond to representations into $D_p$, just as with knots.  Of course, some weightings and colorings correspond to representations {\em onto} $\Gamma(p,m,k)$, while others correspond to representations onto proper subgroups.  We would like to know which weightings and colorings correspond to surjective representations, and to count these representations.  Fox \cite{fo3} addressed this problem for metacyclic representations of knot groups; our goal here is to extend his result to spatial graphs.  

Consider a balanced spatial graph $(G, \omega)$.  If $\omega$ is not reduced, then the weights of $\o$ have a greatest common divisor $d$, and the representation corresponding to any $p$-coloring at $k$ is into the subgroup of $\Gamma(p,m,k)$ generated by $\a$ and $\b^d$.  Depending on whether $d$ and $m$ are relatively prime, this may or may not be a proper subgroup.  However, $(G, \omega)$ has a $p$-coloring at $k$ if and only if $(G, \o/d)$ has a $p$-coloring at $k^d$ (where $\o/d$ is the reduced weighting obtained by dividing every weight in $\o$ by $d$).  So it is reasonable to restrict our attention to reduced weightings.  In this case, our next lemma shows that if a representation into $\Gamma(p,m,k)$ is not surjective, it must be a representation onto a {\em cyclic} subgroup.

\begin{lem} \label{L:cyclic}
Suppose $\f: \pi_1(S^3 - G) \rightarrow \Gamma(p,m,k)$ is a non-trivial representation corresponding to a reduced balanced weighting $\o$ and a $p$-coloring of $(G, \o)$ at $k$.  Then $\f$ is either surjective, or the image is a cyclic subgroup of $\Gamma(p,m,k)$.
\end{lem}
\begin{proof}
Since the weighting is reduced, the image of $\f$ contains $\a^r \b$ for some $r$.  If the image is not the cyclic subgroup generated by $\a^r \b$, then for some $b$ and $x \neq y$, the image contains both $\a^x \b^b$ and $\a^y \b^b$.  But then it contains $\a^x \b^b \b^{-b} \a^{-y} = \a^{x-y} \neq 1$.  Since $\a$ has prime order, the image will then contain $\a$, and hence $\a^{-r} \a^r \b = \b$.  So $\f$ is surjective.  Hence, either $\f$ is surjective, or the image is a cyclic subgroup.
\end{proof}

We now wish to characterize, for a given reduced weighting, the colorings which correspond to cyclic representations into $\Gamma(p,m,k)$.  The image of such a representation is the subgroup generated by $\a^r \b$ for some $r$ ($0 \leq r \leq p-1$).  Observe that, for any $w$,
$$(\a^r \b)^w = \a^{r\left( 1 + k + k^2 + \cdots + k^{w-1}\right)} \b^w = \a^{r\left( \frac{1-k^w}{1-k} \right)} \b^w.$$
So any arc on an edge with weight $w$ has color $r\left( \frac{1-k^w}{1-k} \right)$ (in particular, all arcs on the same edge have the same color).  We need to verify that this coloring satisfies the coloring relations at the crossings and vertices.

Suppose arcs $a$, $b$ and $c$ meet at a crossing as in Figure \ref{F:wirtinger}, and that the edge containing arc $b$ has weight $w_1$ while the edge containing arcs $a$ and $c$ has weight $w_2$.  Using the colors above, we find that
\begin{align*}
(1-k^{w_2})r\left( \frac{1-k^{w_1}}{1-k} \right) &+ k^{w_1}r\left( \frac{1-k^{w_2}}{1-k} \right) - r\left( \frac{1-k^{w_2}}{1-k} \right) \\
&= \frac{r}{1-k} \left( (1-k^{w_1})(1-k^{w_2}) - (1-k^{w_1})(1-k^{w_2}) \right) = 0.
\end{align*}
So the coloring relation is satisfied at the crossing.

Now suppose arcs $a_1, \dots, a_n$ meet at a vertex, where arc $a_i$ is on an edge with weight $w_i$ and has local sign $\e_i$.  Recall that $m_i = \e_1w_1 + \e_2w_2 + \cdots + \e_{i-1}w_{i-1} + \min\{\e_i, 0\}w_i$. An easy induction shows that
$$\sum_{i=1}^l{\e_ik^{m_i}(1-k^{w_i})} = 1-k^{\sum_{i=1}^l{\e_iw_i}}.$$
In particular, since $\sum_{i=1}^n{\e_iw_i} = 0$, this means $\sum_{i=1}^n{\e_ik^{m_i}(1-k^{w_i})} = 0$.  Then
$$\sum_{i=1}^n{\e_i k^{m_i} r\left( \frac{1-k^{w_i}}{1-k} \right)} = \frac{r}{1-k}\sum_{i=1}^n{\e_ik^{m_i}(1-k^{w_i})} = 0$$
and the coloring relation is satisfied at the vertex.  Hence, for each choice of $r$, we get a coloring corresponding to a representation onto a cyclic subgroup of $\Gamma(p,m,k)$.  From section \ref{SS:color}, the number of linearly independent $p$-colorings at $k$ is $N_p(G,\o,k)$, so the total number of representations into $\Gamma(p,m,k)$ is $p^{N_p(G,\o,k)}$, of which $p$ are representations into a cyclic subgroup and the rest are surjective.

To count how many of these representations are inequivalent, we will make a further assumption that $m = {\rm ord}_p(k)$ (rather than just being a multiple).  Then any automorphism of $\Gamma(p,m,k)$ must send $\a$ to $\a^a$ for some $a$ between 1 and $p-1$ (since these are all the elements with order $p$), and must send $\b$ to $\a^b \b$ for some $b$ between $0$ and $p-1$ (since these are the elements which have order $m$ {\em and} satisfy $\b \a = \a^k \b$).  Note that, for any $\b^r$, $\b^r \a = \a^{k^r} \b^r$, and this is not the same as $\a^k \b^r$ unless ${\rm ord}_p(k) \vert (r-1)$; by taking $m = {\rm ord}_p(k)$, we avoid this possibility.  So each equivalence class of representations contains $p(p-1)$ elements.

We summarize this discussion with the following theorem, analogous to Theorem 1 of Fox \cite{fo3}.

\begin{thm} \label{T:countrep}
Given a spatial graph $G$ with a reduced balanced weighting $\o$, the number of inequivalent surjective representations from $\pi_1(S^3-G)$ onto $\Gamma(p, {\rm ord}_p(k), k)$ corresponding to $\o$ is 
$$\frac{p^{N_p(G,\o,k)} - p}{p(p-1)} = \frac{p^{N_p(G,\o,k)-1} - 1}{p-1}$$
\end{thm}

\begin{rem}
Representations of $\pi_1(S^3-G)$ onto the cyclic subgroup generated by $\a$ correspond to the trivial weighting where all weights are 0.  In this case, the crossing relations imply that the colors of two arcs on the same edge are the same.  The vertex relations become $\sum{\e_i c_i} = 0$, where $c_i$ are the colors of the arcs incident to the vertex.  So the possible colorings correspond to the possible balanced weightings of $G$, modulo $p$.
\end{rem}

\section{Properties of the Alexander polynomial} \label{S:properties}

In this section we will examine how the Alexander polynomial changes when we transform the graph in various ways.

\subsection{Reverses and Reflections} \label{SS:orientation}

We will first look at the effect on the Alexander polynomial of reversing the orientation of the graph, or of taking its mirror image.  Figure~\ref{F:reverse} shows the four possible orientations of the graph from Figure \ref{F:bouquet} (with both edges given weight 1) and their first Alexander polynomials; beneath each oriented graph is its mirror image (we imagine reflecting the diagram across a vertical line).  We can see that changing the orientation of a single edge of the graph can substantially change the polynomial.  However, if we reverse the orientation of {\em every} edge, or take the mirror image, we simply reverse the order of the coefficients.  Up to a factor of $\pm t^r$, this is the same as replacing $t$ with $t^{-1}$ in the polynomial.

\begin{figure}[htbp]
\begin{center}
\scalebox{.8}{\includegraphics{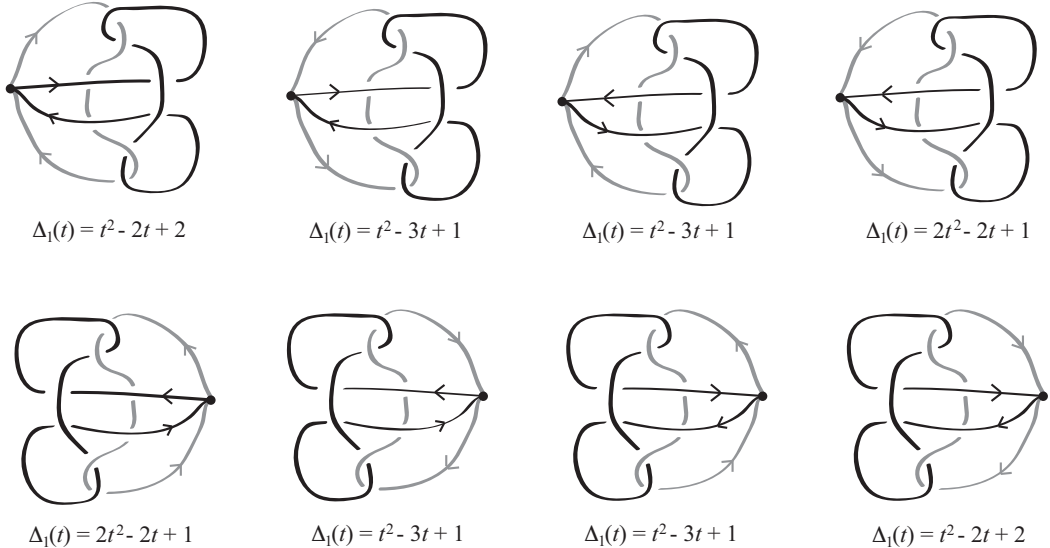}}
\end{center}
\caption{Orientations and mirror images.}
\label{F:reverse}
\end{figure}

\begin{thm}
For any balanced spatial graph $(G, \omega)$, let $(G^*, \omega)$ be the {\em inverse} of $G$ (i.e. the orientation of every edge is reversed), and let $(\overline{G}, \omega)$ be the mirror image (using the same weights as in $G$).  Then $\Delta_k(G^*, \omega)(t) = \Delta_k(\overline{G}, \omega)(t) = \Delta_k(G, \omega)(t^{-1})$.
\end{thm}
\begin{proof}
As we remarked earlier, reversing the orientation of every edge is equivalent to changing the sign of every weight in $\omega$.  But this simply replaces $t$ in the Alexander matrix with $t^{-1}$, so $\Delta_k(G^*, \omega)(t) = \Delta_k(G, \omega)(t^{-1})$.

For knots, the result for mirror images is typically proved using the interpretation of the Alexander polynomial via Seifert surfaces.  Since spatial graphs do not bound surfaces in the same way, we take a more combinatorial approach.  Let $D$ be a diagram for $G$ with $v$ vertices and $c$ crossings, and let $\overline{D}$ be the mirror image of $D$.  We will compare the Alexander relations for $D$ and $\overline{D}$ at the vertices and crossings (we assume corresponding arcs of the two diagrams are given the same label).

At a vertex of $D$ the Alexander relation is $\sum_{i=1}^n{\e_i t^{m_i}a_i}  = 0$ where $m_i = \e_1w_1 + \e_2w_2 + \cdots + \e_{i-1}w_{i-1} + \min\{\e_i, 0\}w_i$.  In $\overline{D}$, the order of the arcs at each vertex has been reversed, but the signs are the same.  The cumulative weight of an arc at a vertex in $\overline{D}$ is $\overline{m_i} = \sum_{j=1}^n{\e_j w_j} - m_i  - \e_i w_i + 2\min\{\e_i, 0\}w_i = -m_i - w_i$ (since $\sum_{j=1}^n{\e_j w_j} = 0$). Then the vertex relation for $\overline{D}$ is:
$$\sum_{i = 1}^n{a_i\e_i t^{\overline{m_i}}} = \sum_{i = 1}^n{a_i\e_i t^{-m_i - w_i}} = \sum_{i = 1}^n{t^{-w_i}a_i\e_i (t^{-1})^{m_i}} = 0$$
If we multiply each column $a_i$ of the Alexander matrix $M(\overline{D}, \omega)$ by $t^{w_i}$, then each minor is changed by a factor of $\pm t^k$, and the vertex rows will then be the same as $M(D, \omega)$ with $t$ replaced by $t^{-1}$.

Now consider a crossing of $D$, where the relation is $(1-t^{w_2})b + t^{w_1}c - a = 0$ (arc $b$ has weight $w_1$ and arcs $a$ and $c$ both have weight $w_2$).  The relation at the corresponding crossing of $\overline{D}$ (reflecting across a vertical line) is $(1-t^{w_2})b - c + t^{w_1} d = 0$, since $a$ and $c$ have switched sides.  When each column of $M(\overline{D}, \omega)$ is multiplied by $t^{w_i}$ (to adjust the vertex relations, as described in the last paragraph), this relation becomes:
$$(t^{w_1} - t^{w_1+w_2})b - t^{w_2} c + t^{w_1 + w_2} a = -t^{w_1+w_2}((1-t^{-w_2})b + t^{-w_1}c - a) = 0$$
This is the same as the relation for the crossing in $D$, except with $t$ replaced by $t^{-1}$, and the relation multiplied by $-t^{w_1+w_2}$.  So we conclude that the minors of $M(\overline{D}, \omega)$ are the same as those of $M(D, \omega)$, except with $t$ replaced by $t^{-1}$, up to multiplication by a power of $t$.  Hence $\Delta_k(\overline{G}, \omega)(t) = \Delta_k(G, \omega)(t^{-1})$.
\end{proof}

\begin{rem}
Unlike for knots, where the coefficients of the Alexander polynomial are always palindromic, our weighted Alexander polynomial can sometimes distinguish an oriented spatial graph from its inverse and its mirror image.  For example, the oriented knot at the top left in Figure \ref{F:reverse} is both non-invertible and chiral.
\end{rem}

\subsection{Graph operations} \label{SS:operations}

Now we will explore how the Alexander polynomials of a spatial graph are affected by operations on the graph.  We will consider two common operations: contracting an edge in a graph, and joining two graphs at a vertex.

\begin{figure}[htbp]
\begin{center}
\scalebox{.8}{\includegraphics{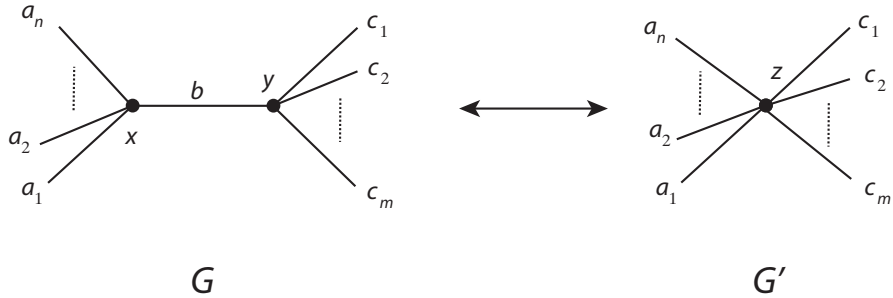}}
\end{center}
\caption{Contracting an edge.}
\label{F:contractedge}
\end{figure}

We first look at the effect of contracting an edge, as in Figure \ref{F:contractedge}.  We will show this does not change the Alexander modules, and hence does not change the Alexander polynomials.

\begin{thm}\label{T:alexcontract}
Let $(G, \omega)$ be a balanced spatial graph.  If $G'$ is the result of contracting an edge of $G$, as in Figure \ref{F:contractedge}, with the weighting $\omega'$ and orientation inherited from $G$, then $A(G', \omega') \cong A(G, \omega)$.  
\end{thm}

\begin{proof}
The complements of $G$ and $G'$ are homeomorphic, so their fundamental groups are isomorphic.  We will show that their augmentation homomorphisms induce isomorphic covering spaces, and hence isomorphic Alexander modules.  The edges of $G'$ are given the same orientations and weights as the corresponding edges in $G$.  Suppose arc $a_i$ has weight $w_i$ and local sign $\e_i$, arc $c_i$ has weight $u_i$ and local sign $\f_i$, and arc $b$ has weight $w$ and local sign $\e$ at $x$ and $-\e$ at $y$.  

The Wirtinger presentations of their fundamental groups are almost identical.  We will also use $a_i$, $c_i$ and $b$ to refer to the generators corresponding to those arcs in $G$. We suppose the generators for $a_i$ and $c_i$ are chosen assuming the edges are oriented towards $x$ and $y$, respectively, and that the generator for $b$ is chosen assuming $b$ is oriented towards $x$. Then the presentation for the fundamental group of $G$ has two relations $a_1^{\e_1}\cdots a_n^{\e_n} b^\e = 1$ and $c_1^{\f_1}\cdots c_m^{\f_m} b^{-\e} = 1$; whereas $G'$ has only the single relation $a_1^{\e_1}\cdots a_n^{\e_n}c_1^{\f_1}\cdots c_m^{\f_m} = 1$ and does not include $b$ as a generator (all other relations and generators are the same).  An isomorphism $f$ from $\pi_1(S^3 - G)$ to $\pi_1(S^3-G')$ is defined by $f(a_i) = a_i$, $f(c_i) = c_i$, and $f(b) = (c_1^{\f_1}\cdots c_m^{\f_m})^{1/\e}$. We observe:
\begin{align*}
\chi(b) &= t^w \\
\chi'(f(b)) &= \chi'((c_1^{\f_1}\cdots c_m^{\f_m})^{1/\e}) = (t^{\f_1 u_1}\cdots t^{\f_m u_m})^{1/\e} = t^{\left(\sum{\f_i u_i}\right)/\e}
\end{align*}

\noindent Since the weighting on $G$ is balanced, we have:
$$-\e w + \sum_{i = 1}^m{\f_i u_i} = 0 \implies w = \frac{1}{\e}\left(\sum_{i = 1}^m{\f_i u_i}\right)$$
Hence $\chi' \circ f = \chi$.  Since $f$ is an isomorphism, ${\rm Ker}(\chi') \cong {\rm Ker}(\chi)$, and hence the induced covering spaces are isomorphic.  Therefore, $A(G', \omega') \cong A(G, \omega)$.
\end{proof}

\begin{cor} \label{C:contract}
If $G'$ is the result of contracting an edge of $G$, as in Theorem \ref{T:alexcontract}, then $\Delta_k(G', \omega')(t) = \Delta_k(G, \omega)(t)$. 
\end{cor}

As an application, we will analyze what happens to the Alexander polynomial when we replace each edge of a graph with $n$ parallel edges with the same weight (or, more generally, $k$ parallel edges with one orientation and $n-k$ parallel edges with the opposite orientation).

\begin{figure}[htbp]
\begin{center}
\scalebox{.8}{\includegraphics{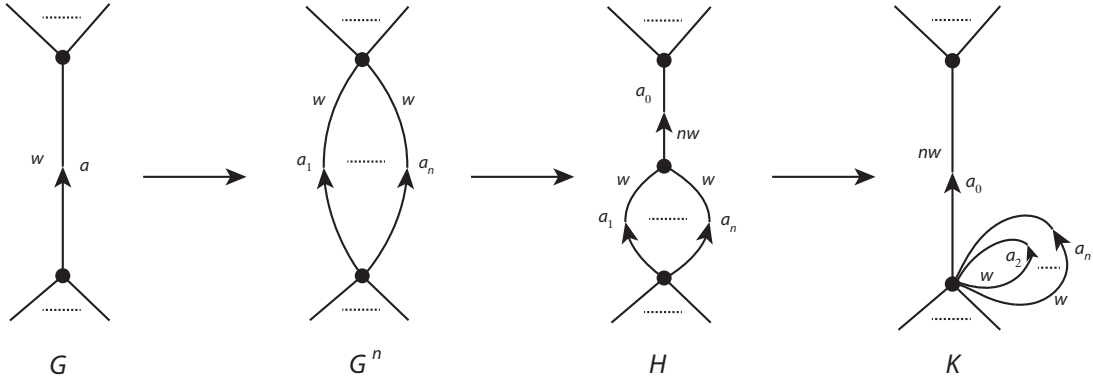}}
\end{center}
\caption{Replacing an edge with a ``bundle" of parallel edges.}
\label{F:parallel}
\end{figure}

\begin{thm} \label{T:parallel}
Let $(G, \omega)$ be a balanced spatial graph, and $(G^n, \omega^n)$ the result of replacing each edge with $n$ parallel copies of the edge, all with the same orientation and weight as the original edge (see Figure \ref{F:parallel}).  Then $\Delta_k(G^n, \omega^n)(t) = \Delta_k(G, \omega)(t^n)$.  If the edges are oriented so that $r$ have the same orientation as the original edge, and $n-r$ have the opposite orientation, denote the graph by $(G^n_r, \omega^n)$.  Then $\Delta_k(G^n_r, \omega^n)(t) = \Delta_k(G, \omega)(t^{r-(n-r)}) = \Delta_k(G, \omega)(t^{2r-n})$.
\end{thm}

\begin{proof}
We will prove the theorem for $G^n$; the proof for $G^n_r$ is the same.  For each ``bundle" of parallel edges in $G^n$, we can expand the vertex at one end of the bundle to obtain a new graph $H$, as shown in Figure \ref{F:parallel}.  We expand the vertex by adding a single edge, with $n$ times the weight of each edge in the bundle, so the graph is still balanced.  By Theorem \ref{T:alexcontract}, $G^n$ and $H$ have isomorphic Alexander modules.  We continue by contracting one of the edges in the bundle to get a graph $K$ (with the same Alexander module), as in Figure \ref{F:parallel}.  The other edges of the bundle are now small loops at one vertex, which we may assume involve no crossings.  Each of these arcs contributes a column of 0's to the presentation matrix for the Alexander module, and hence they do not change the Alexander polynomials.   Ignoring these loops, the presentation matrix for $K$ is the same as the matrix for $G$, except that each weight is multiplied by a factor of $n$.  This is equivalent to replacing $t$ with $t^n$ in the matrix for $G$, and so $\Delta_k(G^n, \omega^n)(t) = \Delta_k(G, \omega)(t^n)$.
\end{proof}

\begin{rem}
The proof of Theorem \ref{T:parallel} implies that, in any balanced spatial graph, an edge of weight $w$ can be replaced by $w$ parallel edges of weight 1 (and vice versa) without changing the Alexander polynomial.
\end{rem}

\begin{figure}[htbp]
\begin{center}
\scalebox{.7}{\includegraphics{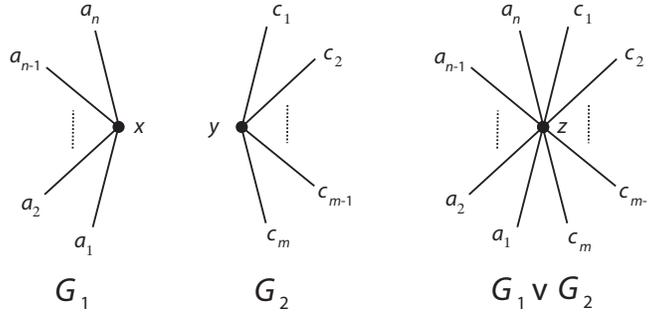}}
\end{center}
\caption{Joining two graphs at a vertex.}
\label{F:vertexjoin}
\end{figure}

Our next result looks at the result of taking the wedge product of two graphs, as shown in Figure~\ref{F:vertexjoin}.  In this case, we will show that the Alexander polynomial is, in a sense, multiplicative.  We will work with the presentation matrices for the Alexander modules; the following well-known facts from linear algebra will be useful (see, for example, \cite[Ex. 7.2.4]{ka}).  We will sketch the main idea of the proofs; the details are left to the reader.

\begin{lem}\label{L:operations}
Let $M$ be a matrix with entries in $\Z[t^{\pm 1}]$.  Then interchanging two rows (columns) or adding an multiple of a row (column) to another row (column) (where the scalars are in $\Z[t^{\pm 1}]$) do not change $\det(M, k)$.
\end{lem}
\begin{proof}
Interchanging two rows (columns) of the matrix only changes some minors by a sign, so does not affect the greatest common divisor.  Adding a multiple of a row (column) to another replaces some pairs of minors $\{a, b\}$ with pairs $\{a + nb, b\}$; again, the greatest common divisor of the minors is unchanged.
\end{proof}

\begin{lem}\label{L:blockmatrix}
Suppose that $M = \begin{pmatrix}
[M_1]&0&0&0 \\
0&[M_2]&0&0 \\
0&0&\ddots&0 \\
0&0&0&[M_n] \\
\end{pmatrix}$ ($M_1$, $M_2$ and $M_n$ do not have to be the same dimensions, or square).  Then 

\begin{center} $\det(M, k) = \gcd\{\det(M_1, k_1) \det(M_2, k_2) \dotsm \det(M_n, k_n) \vert k_1 + k_2 + \cdots + k_n = k\}$ \end{center}
In particular, if $M_i = [\pm 1]$ for $2 \leq i \leq n$, then $\det(M) = \det(M_1, k-(n-1))$.
\end{lem}
\begin{proof}
Any $k \x k$ minor of $M$ which contains more rows than columns (or vice versa) of a single block $M_i$ will have determinant 0, so we only need to consider minors constructed from square minors of each block.
\end{proof}

\begin{thm} \label{T:wedge}
If $(G_1, \omega_1)$ and $(G_2, \omega_2)$ are balanced spatial graphs, and $G = G_1 \vee G_2$ (i.e. the result of joining $G_1$ and $G_2$ at a vertex), with every edge inheriting its orientation and weighting from $G_1$ and $G_2$, then 
$$\Delta_k(G, \omega) = \gcd\{\Delta_{k_1}(G_1, \omega_1)\Delta_{k_2}(G_2, \omega_2) \vert k_1 + k_2 = k+1;\ k_1, k_2 \geq 1\}$$
In particular, $\Delta_1(G, \omega) = \Delta_1(G_1, \omega_1)\Delta_1(G_2, \omega_2)$.  
\end{thm}
\begin{proof}
We consider diagrams $D_1$ and $D_2$ for two graphs $G_1$, with vertex $x$, and $G_2$, with vertex $y$, as shown in Figure~\ref{F:vertexjoin}.  $G_1 \vee G_2$ is the result of joining these graphs at vertices $x$ and $y$, creating a new vertex $z$ and diagram $D_1 \vee D_2$, as shown.  We assume this is done so that there is a 3-ball $B \subset S^3$ with $\partial B \cap (G_1 \vee G_2) = \{z\}$, $G_1 - \{x\}$ is in the interior of $B$, and $G_2 - \{y\}$ is in the exterior of $B$.  In particular, there are no crossings between the edges of $G_1$ and the edges of $G_2$. If $D_i$ has $c_i$ crossings and $v_i$ vertices, then $D_1 \vee D_2$ has $c_1+c_2$ crossings and $v_1+v_2-1$ vertices.

In the Wirtinger presentation, any one relation is a consequence of the others, and can be deleted from the presentation.  If we delete the relations at vertex $x$ in $\pi_1(S^3 - G_1)$, at vertex $y$ in $\pi_1(S^3 - G_2)$ and at vertex $z$ in $\pi_1(S^3 - (G_1 \vee G_2))$, then we see that $\pi_1(S^3 - (G_1 \vee G_2))$ is the free product of $\pi_1(S^3 - G_1)$ and $\pi_1(S^3 - G_2)$.  Hence the Alexander module for $G_1 \vee G_2$ is the direct sum of the modules for $G_1$ and $G_2$.  So with the proper labeling, and using Lemma \ref{L:operations} to remove the redundant rows, $M(D_1 \vee D_2, \omega)$ is a block matrix whose blocks are $M(D_1, \omega_1)$ and $M(D_2, \omega_2)$.  Then by Lemma \ref{L:blockmatrix}:
\begin{align*}
	\Delta_k(G_1 \vee G_2, \omega) &= \det(M(D_1 \vee D_2, \omega), c_1 + c_2 + v_1 + v_2 -1 -k) \\
	&= \gcd\{\det(M(D_1, \omega_1), c_1 + v_1 - k_1)\det(M(D_2, \omega_2), c_2 + v_2 - k_2) \vert k_1 + k_2 = k+1;\ k_1, k_2 \geq 1\} \\
	&= \gcd\{\Delta_{k_1}(G_1, \omega_1)\Delta_{k_2}(G_2, \omega_2) \vert k_1 + k_2 = k+1;\ k_1, k_2 \geq 1\}
\end{align*}
  
In particular, if $k = 1$, then the only possible choice with $k_1 + k_2 = 2$ is $k_1 = k_2 = 1$. Hence $\Delta_1(G_1\vee G_2, \omega) = \Delta_1(G_1, \omega_1)\Delta_1(G_2, \omega_2)$.
\end{proof}

\section{Questions}

The Alexander polynomial for knots is well known to be {\em palindromic}, meaning that (modulo $\pm t^r$) it is symmetric in $t$ and $t^{-1}$.  Even more, this (and a few more minor conditions) characterizes the polynomials that can be realized as the Alexander polynomial of a knot.  However, the graphs in Figure \ref{F:reverse} show that the Alexander polynomial of a spatial graph may not be palindromic, so they are not characterized by the same conditions.  Corollary \ref{C:alex(1)} tells us that the sum of the coefficients must equal 1, but we do not know if this condition is sufficient to characterize the Alexander polynomials of balanced spatial graphs.

\begin{question}
Which polynomials can be realized as the Alexander polynomial of a balanced spatial graph?
\end{question}

The Alexander polynomial for a knot can also be transformed into the {\em Conway polynomial} by the substitution $z = t^{-1/2} - t^{1/2}$, in which form the coefficients are {\em finite type invariants} of the knot and the polynomial satisfies a nice skein relation.  This relation is useful for computing the polynomial and proving its properties.

\begin{question}
Is there a ``Conway normalization" for the Alexander polynomial of spatial graphs?  In other words, does the polynomial satisfy a nice skein relation?
\end{question}

\subsection*{Acknowledgements}
The authors are very grateful to Dan Silver for many helpful comments and discussions.

\small

\normalsize

\end{document}